\documentclass[12pt,oneside,a4paper]{amsart}

\usepackage[utf8]{inputenc}
\usepackage{mathtools}
\usepackage{enumerate}

\usepackage{amsthm}
\usepackage{amssymb}
\usepackage{thmtools}

\usepackage{nameref}
\usepackage{hyperref}
\usepackage[capitalize]{cleveref}

\declaretheorem[numberwithin=section]{theorem}
\declaretheorem[numbered=no,name=Theorem]{theorem*}
\declaretheorem[numberlike=theorem]{lemma}
\declaretheorem[numberlike=theorem]{corollary}
\declaretheorem[numberlike=theorem]{proposition}

\declaretheorem[style=definition,numberlike=theorem]{definition}

\declaretheorem[style=remark,numberlike=theorem]{remark}

\newcommand{\engel}{{\operatorname{\mathbf{E}}\nolimits}}
\newcommand{\cartan}{{\operatorname{\mathbf{C}}\nolimits}}
\newcommand{\optimalpreimage}{\widehat{N}}
\newcommand{\preimagedomain}{\widetilde{N}}
\newcommand{\imagedomain}{\widetilde{\cartan}}
\newcommand{\identity}{\mathrm{Id}}
\def\tcutc{\tcut^\cartan}
\def\tcute{\tcut^\engel}
\newcommand{\qzero}{\identity}
\newcommand{\qone}{\mathbf{q}}

\newcommand{\Exp}{\operatorname{Exp}\nolimits}

\newcommand{\MAX}{\operatorname{MAX}\nolimits}
\newcommand{\FIX}{\operatorname{FIX}\nolimits}
\newcommand{\Cut}{\operatorname{Cut}\nolimits}

\def\tmax{t_{\MAX}^1}
\def\tcut{t_{\operatorname{cut}}}
\def\tconj{t_{\operatorname{conj}}^1}
\newcommand{\Sym}{\mathbf{S}}
\newcommand{\Id}{\operatorname{Id}\nolimits}
\def\SE{\mathrm{SE}}
\def\SH{\mathrm{SH}}
\def\SO{\mathrm{SO}}
\def\SL{\mathrm{SL}}

\def\H{\mathrm{H}}

\newcommand{\quotientexp}{\widetilde{\Exp}}
\newcommand{\bpiE}{\bar{\pi}_{\engel}}
\newcommand{\E}{\operatorname{E}\nolimits}
\newcommand{\RR}{\mathbb{R}}
\newcommand{\NN}{\mathbb{N}}
\DeclarePairedDelimiter{\abs}{\lvert}{\rvert}
\DeclarePairedDelimiter{\norm}{\lVert}{\rVert}

\newcommand{\sn}{\operatorname{sn}}
\newcommand{\cn}{\operatorname{cn}}
\newcommand{\dn}{\operatorname{dn}}

\newcommand{\Ad}{\operatorname{Ad}}
\makeatletter
\let\oldabs\abs
\def\abs{\@ifstar{\oldabs}{\oldabs*}}
\let\oldnorm\norm
\def\norm{\@ifstar{\oldnorm}{\oldnorm*}}
\makeatother

\makeatletter
\def\@tocline#1#2#3#4#5#6#7{\relax
	\ifnum #1>\c@tocdepth 
	\else
	\par \addpenalty\@secpenalty\addvspace{#2}%
	\begingroup \hyphenpenalty\@M
	\@ifempty{#4}{%
		\@tempdima\csname r@tocindent\number#1\endcsname\relax
	}{%
		\@tempdima#4\relax
	}%
	\parindent\z@ \leftskip#3\relax \advance\leftskip\@tempdima\relax
	\rightskip\@pnumwidth plus4em \parfillskip-\@pnumwidth
	#5\leavevmode\hskip-\@tempdima
	\ifcase #1
	\or\or \hskip 1em \or \hskip 2em \else \hskip 3em \fi%
	#6\nobreak\relax
	\hfill\hbox to\@pnumwidth{\@tocpagenum{#7}}\par
	\nobreak
	\endgroup
	\fi}
\makeatother

\title[Cut time in the sub-Riemannian Cartan group]{Cut time in the sub-Riemannian problem on the Cartan group}
\author{Andrei Ardentov}
\address[Ardentov]{Ailamazyan Program Systems Institute, Russian Academy of Sciences, Pereslavl-Zalessky, Russia}
\email{aaa@pereslavl.ru}
\author{Eero Hakavuori}
\address[Hakavuori]{SISSA, Via Bonomea 265, 34136 Trieste, Italy}
\email{eero.hakavuori@sissa.it}
\date{July 14, 2021}

\begin{document}
\begin{abstract}
	We study the sub-Riemannian structure determined by a left-invariant distribution of rank 2 on a step 3 Carnot group of dimension 5.
	We prove the conjectured cut times of Y. Sachkov for the sub-Riemannian Cartan problem.
	Along the proof, we obtain a comparison with the known cut times in the sub-Riemannian Engel group, and a sufficient (generic) condition for the uniqueness of the length minimizer between two points. Hence we reduce the optimal synthesis to solving a certain system of equations in elliptic functions.
\end{abstract}
	
\maketitle
\tableofcontents

\section{Introduction}
\subsection{Background}
The sub-Riemannian Cartan group $\cartan$ is the nilpotent model for all sub-Riemannian problems with growth vector (2,3,5). 
As a consequence of the non-integrability results of \cite{BBKM:2016:nonintegrability-rank3-step3} and \cite{LokutsievskiiSachkov:2018:liouville-step4}, it is the only free nilpotent group with step three or greater for which the Hamiltonian system of the Pontryagin Maximum Principle (PMP) \cite{PBGM} is Liouville integrable.

A geometric description of the optimal control problem in the sub-Riemannian Cartan group can be given in terms of the generalized (dual) Dido problem: Given two points $a,b\in\RR^2$ and a fixed ``shoreline'', i.e., a~curve $\bar{\gamma}$ connecting $b$ and $a$, fix a desired oriented area $S\in\RR$ and a desired center of mass $c\in\RR^2$, see Fig.~\ref{geom-statement}. The problem is to find the shortest curve $\gamma$ connecting $a$ and $b$ such that the region of the plane bounded by the curves $\bar{\gamma}$ and $\gamma$ has area equal to $S$ and center of mass $c$.

\begin{figure}[tbhp]
\includegraphics[width=0.55\textwidth]{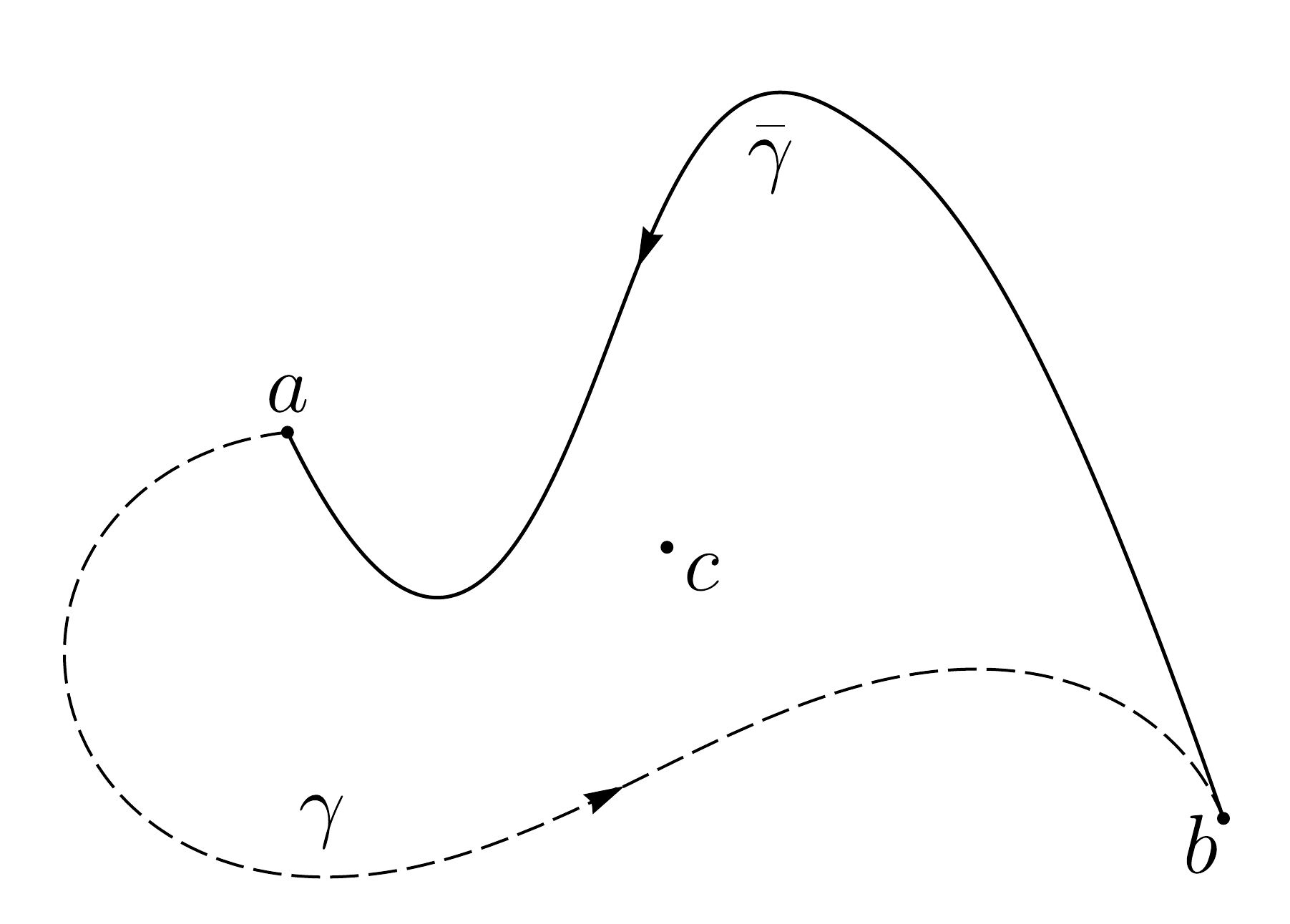}
\caption{Geometric formulation of the problem}
\label{geom-statement}
\end{figure}

In this paper, we consider the equivalent simplified problem where $\bar{\gamma}$ is the straight line connecting $b$ and $a$. 

A closely related problem is that of optimal control in the sub-Riemannian Engel group $\engel$ with the growth vector $(2,3,4)$. The geometric description of the Engel problem is the same as in the Cartan case, except instead of fixing the center of mass, we fix a line on which the center of mass should lie.

In both the Cartan and Engel groups, application of the PMP leads to a complete description of the geodesics. This is due to the fact that all the (injective) abnormal trajectories are straight lines, so there are no strictly abnormal trajectories. The normal extremal trajectories in both the Cartan and Engel cases project to Euler elasticae in the plane. Conversely, every elastica lifts to an extremal trajectory in the Cartan group and, if suitably rotated, also in the Engel group \cite{Sachkov:2003:exp_map,ArdentovSachkov:2011:ExtremalEngel}.

A key part of the optimal synthesis is to understand when each extremal trajectory loses its global optimality, i.e., to understand the cut times. In this paper, we consider all extremal trajectories to be parametrized by arc length.

\begin{definition}
	Let $q\colon [0,\infty)\to M$ be an arc length parametrized extremal trajectory in a sub-Riemannian manifold $M$. The \emph{cut time} $\tcut(q)$ is the maximal time $T$ such that $q\colon [0,T]\to M$ is a minimizing geodesic.
\end{definition}

When $M$ is the Cartan group $\cartan$ or the Engel group $\engel$, we will also use the notation $\tcutc(\gamma)$ or $\tcute(\gamma)$, where the elastica $\gamma\colon [0,\infty)\to\RR^2$ is the projection of a Cartan or Engel extremal trajectory.

The Euler elasticae can be inflectional, non-inflectional, critical, or straight lines. 
For both the Cartan and Engel problems, straight lines and critical elasticae are optimal for all time, so the study of the cut time reduces to studying lifts of inflectional and non-inflectional elasticae.

Upper bounds for the cut time are provided by Maxwell times.

\begin{definition}
	Let $q\colon [0,\infty)\to M$ be an extremal trajectory in a sub-Riemannian manifold $M$.
	A point $q(t)$ is called a \emph{Maxwell point} if there exists another extremal trajectory $\tilde{q}\neq q$ such that $\tilde{q}(t) = q(t)$. The instant $t$ is called a \emph{Maxwell time}.
\end{definition}

In the Engel case $M=\engel$, the cut times are solved in \cite{ArdentovSachkov:2015:CutEngel}. A key ingredient to identify the cut times is the description of a discrete group of symmetries and their Maxwell times \cite{ArdentovSachkov:2011:ExtremalEngel}.

For the Cartan case $M=\cartan$, the analogous discrete dihedral group of symmetries is described in~\cite{Sachkov:2006:dido_discrete_symmetries}. This group of symmetries is generated by a symmetry $\varepsilon^1$, which reflects an elastica in the center of its chord, and a symmetry $\varepsilon^2$, which reflects an elastica in the perpendicular bisector to the chord (up to an additional rotation).
A detailed description of the Maxwell times corresponding to these symmetries is given in \cite{Sachkov:2006:dido_complete_description_of_maxwell_strata}. Based on numerical evidence, it was also conjectured that these times are in fact the first Maxwell times.

In the Engel case, a similar cut time conjecture was proved in two steps:
\begin{enumerate}[1)]
	\item Study the first conjugate times \cite{ArdentovSachkov:2013:ConjugateEngel}.
	\item Prove uniqueness of the geodesics connecting the initial point with every point before the obtained Maxwell point (or the first conjugate point) \cite{ArdentovSachkov:2015:CutEngel}.
\end{enumerate}
In the Cartan case, we use the same steps to validate the conjectured cut times.
For the first step, bounds for the conjugate times have been obtained recently in \cite{Sachkov:2021:Cartan_conj_time}.
Our main goal is to complete the second step and hence obtain the cut times in the Cartan case.

\subsection{Main results}

\begin{theorem}\label{theorem-unique-minimizer-region}
	In the generalized Dido problem, if the desired enclosed area $S$ is nonzero and the center of mass does not lie on the perpendicular bisector to the line segment from $a$ to $b$,
	then there exists a unique minimizer connecting the points $a$ and $b$.
\end{theorem}

In \cite{ArdentovSachkov:2015:CutEngel}, the analogous result is proved by a case-by-case study of an explicit parametrization of the geodesics. In the Cartan case, we are able to give a sufficient reduction of the technical analysis so that one case follows from continuity of the parametrization (see \cref{lemma:abnormal limit}), and the most difficult case can be obtained using the results of \cite{ArdentovSachkov:2015:CutEngel} for the Engel group (see \cref{lemma-unbdd-time}).
See also \cref{theorem-unique-minimizer-region-coords} for an alternate description of \cref{theorem-unique-minimizer-region} using coordinates on the Cartan group.

Consequently we verify the conjectured cut time of \cite{Sachkov:2006:dido_complete_description_of_maxwell_strata} in \cref{tcut-is-tbold}. That is, we obtain the following.

\begin{theorem}\label{theorem-cuttime-is-maxwell-time}
	For every extremal $q$ in the sub-Riemannian Cartan group, the cut time is equal to the first Maxwell time of $q$ or the limit of Maxwell times of extremals $q_n$ converging to $q$.
	Moreover, the first Maxwell time for each extremal is the first Maxwell time corresponding to one of the symmetries $\varepsilon^1$ or $\varepsilon^2$.
\end{theorem}

In the case when the cut time is equal to the limit of Maxwell times of nearby extremals, the results of \cite{Sachkov:2021:Cartan_conj_time} imply that the cut time is in fact equal to the first conjugate time.

\begin{definition}
	Let $q\colon [0,\infty)\to M$ be a geodesic in a sub-Riemannian manifold $M$. For any $\varphi\in[0,\infty)$, let $q_\varphi\colon [0,\infty)\to M$ be the geodesic $q_\varphi(t) = q(t+\varphi)$. The geodesic $q$ is called \emph{equioptimal} if $\tcut(q) = \tcut(q_\varphi)$ for all $\varphi\in[0,\infty)$.
	The sub-Riemannian manifold $M$ is called \emph{equioptimal}, if all of its arc length parameterized geodesics are equioptimal.
\end{definition}

As a consequence of the formula for the cut times in the Engel case in \cite{ArdentovSachkov:2015:CutEngel}, the Engel group is equioptimal. The analogous equioptimality result in the Cartan case is given in \cref{Cartan-invariant-properties}.

Every arc of an elastica that is optimal for the Engel problem is also optimal for the Cartan problem, so the Cartan cut times are never smaller than the Engel cut times. We show that, up to a constant factor, there is also a converse bound.

\begin{theorem}\label{theorem-cuttime-comparison}
	There exists a constant $\zeta<2$ such that for any elastica $\gamma\colon [0,\infty)\to\RR^2$ that is rotated so that it lifts to an extremal trajectory for the Engel group, we have
	\begin{equation*}
	\tcut^{\engel}(\gamma)\leq \tcutc(\gamma) \leq \zeta\cdot \tcute(\gamma).
	\end{equation*}
\end{theorem}

See \cref{fig-optimal-family} for a visual description of the longest optimal arcs of inflectional and non-inflectional elasticae starting from the point of minimum absolute curvature. The values $k_1$ and $k_0$ refer to certain critical values of the parametrization of inflectional elasticae, see \cref{subsec-cartan} for details.

\begin{figure}[htbp]
\includegraphics[width=0.99\textwidth]{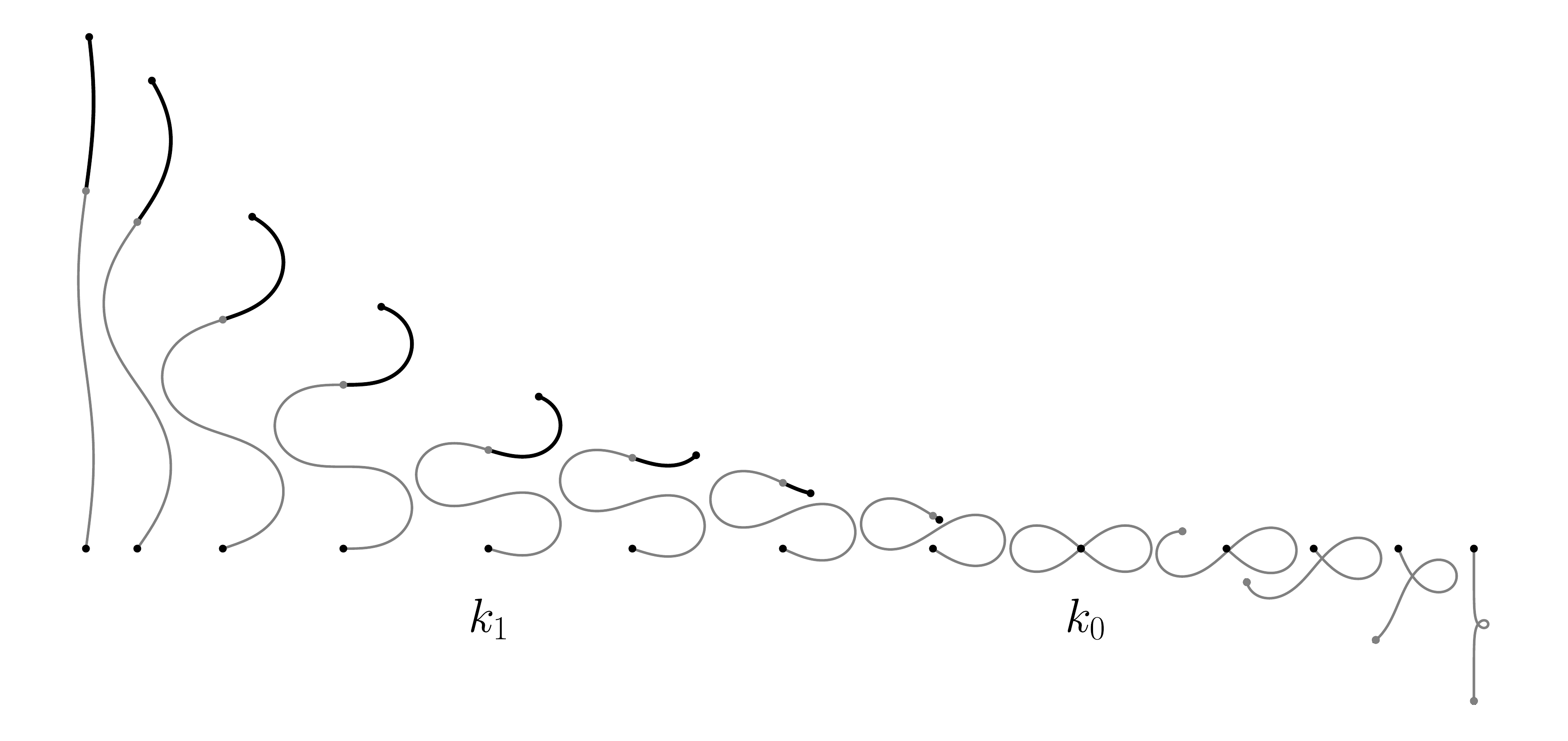}
\includegraphics[width=0.99\textwidth]{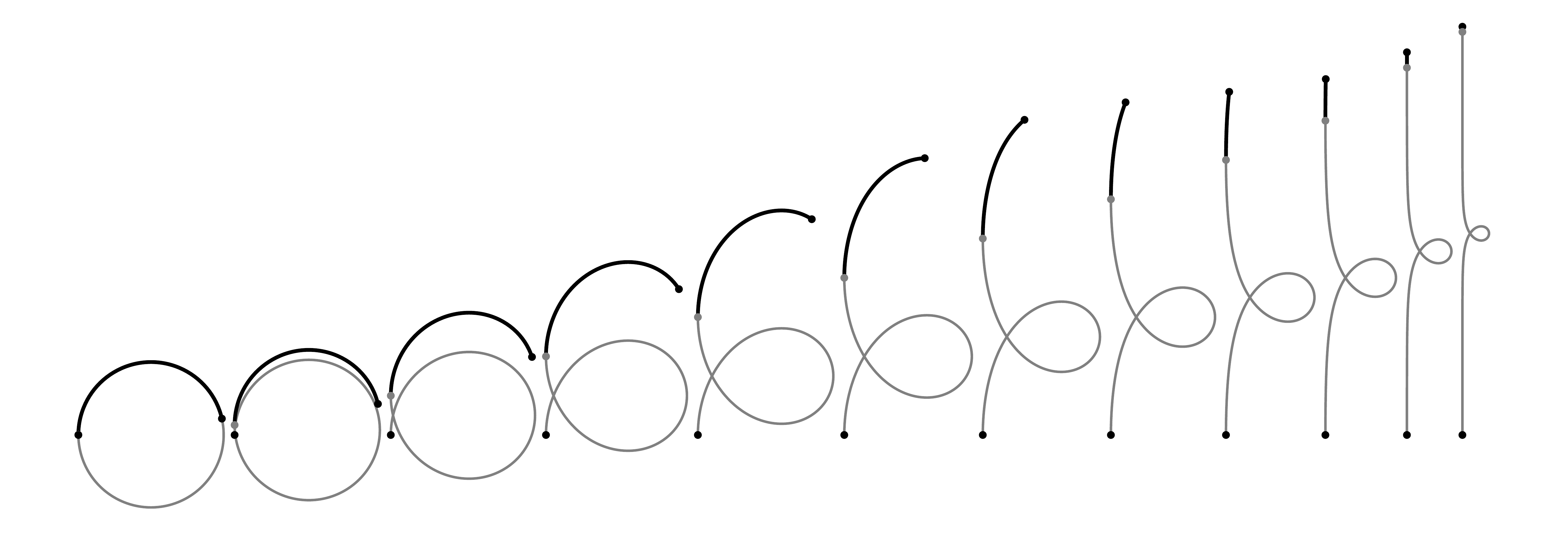}
\caption{Family of longest optimal inflectional and non-inflectional elasticae in the Engel case (gray arcs), and in the Cartan case (unions of gray and black arcs)}
\label{fig-optimal-family}
\end{figure}

\subsection{Structure of the paper}
In \cref{section-known-results}, we cover the basic definitions of sub-Riemannian geometry and cover relevant known results for the Cartan and Engel groups.
In \cref{section-properness}, we prove our most important technical result, \cref{prop-exp-is-proper}, stating that a certain restriction of the sub-Riemannian exponential mapping is proper.
In \cref{section-cut-time}, we obtain our main results and describe properties of the Cartan and Engel cut times along with their visual comparison.
We discuss some related open problems in \cref{section-open-questions}.
In \cref{appendix-formulas}, we give formulas for the first Maxwell times as the roots of certain equations depending on elliptic functions.

\section{Preliminaries}\label{section-known-results}

In this section, we give the formulation of the sub-Riemannian problems on the Cartan and Engel groups and describe previously obtained results for both problems.
\subsection{Optimal control problem}
A left-invariant sub-Riemannian problem on a Lie group $M$ with two-dimensional control $(u_1, u_2)\in\RR^2$ can be formulated as follows:
\begin{align}
\dot{q} &= u_1 X_1 + u_2 X_2, \qquad q\in M,\label{sys}\\
q(0) &=\Id, \qquad q(T)=\qone, \label{points} \\
l(q(\cdot))&=\int_0^T \sqrt{u_1^2 +u_2^2} \, d t \to \min, \label{integralSR}
\end{align}
where the vector fields $X_1$ and $X_2$ on $M$ are left-invariant and generate the Lie algebra of $M$; the terminal time $T$ is not fixed. By left-invariance, there is no loss of generality in assuming that the initial point $q(0)$ is the identity $\identity$.

The solutions of the problem \eqref{sys}--\eqref{integralSR} define the sub-Riemannian distance as $d(\qzero,\qone)=l(q(\cdot))$, where $q$ is the optimal curve connecting $\identity$ with $\qone$.
The optimal control $(u_1(t),u_2(t))$ and the desired trajectory $q(t)$ can have arbitrary time parametrization, so, without loss of generality, we assume that all solutions are parametrized with constant speed $\sqrt{u_1^2+u_2^2} \equiv \text{const}$. By the Cauchy-Schwarz inequality, it follows that the sub-Riemannian length minimization problem~(\ref{integralSR}) is equivalent to the action minimization problem with a fixed terminal time $T$:
\begin{align}
&\int_0^T \frac{u_1^2 +u_2^2}{2} \, d t \to \min. \label{integralE}
\end{align}

As mentioned in the introduction, application of the PMP to problem~(\ref{sys}), (\ref{points}), (\ref{integralE}) leads to a description of the geodesics. When $M$ is the Cartan or the Engel group, all abnormal geodesics are simultaneously normal, so we consider only normal geodesics.  

Normal geodesics are solutions to the Hamiltonian system
\begin{align}
&\dot{\lambda} = \vec{H} (\lambda), \qquad \lambda \in T^* M, \label{Hamsys}
\end{align}
given by the maximized Hamiltonian function 
\begin{equation*}
H(\lambda) = \frac12 \big(h_1^2(\lambda)+h_2^2(\lambda)\big), \qquad h_i(\lambda) = \langle \lambda, X_i \rangle, \ i =1,2,
\end{equation*}
with normal extremal controls $u_i(\lambda) = h_i(\lambda)$.

Arc length parameterized geodesics are projections of extremals $\lambda_t$ with $H(\lambda_t)\equiv 1/2$. The initial cylinder is defined by
\begin{align}
C = \left\{\lambda \in T_{\qzero}^* M \mid H(\lambda)= 1/2 \right\}. \label{cylinder}
\end{align}
Integration of~(\ref{Hamsys}) gives the parametrization of all (arc length parametrized) extremal trajectories, defining the exponential mapping
\begin{align*}
&\Exp \colon N = C \times [0,\infty) \to M,\\
&\Exp (\lambda,t) = q_t.
\end{align*}

\begin{definition}
\emph{The cut time} $\tcut (\lambda)$ is the time when the extremal trajectory corresponding to the covector $\lambda$ loses its global optimality:
\begin{align*}
&\tcut(\lambda) = \sup \{T > 0 \mid \Exp(\lambda, t) \text{ is optimal for } t \in [0, T]\}.
\end{align*}
\end{definition}

\begin{definition}
A point $q_T = \Exp(\lambda, T)$ is called \emph{a conjugate point} for the point $\qzero$ if $\nu = (\lambda, T)$ is a critical point of the exponential mapping.
The instant $T$ is then called \emph{a conjugate time} along the extremal trajectory $q_t = \Exp(\lambda, t), t\in[0,\infty)$.
\end{definition}

\begin{definition}
A point $q_T$ of an extremal trajectory $q_t = \Exp(\lambda, t)$ is called \emph{a Maxwell point} if there exists another extremal trajectory $\tilde{q}_t = \Exp(\tilde{\lambda}, t), q_t \not\equiv \tilde{q}_t$, such that $\tilde{q}_T = q_T$. The instant $T$ is called \emph{a Maxwell time}.
\end{definition}

We denote the first conjugate time by $\tconj(\lambda)>0$ and the first Maxwell time by $\tmax(\lambda)>0$ for the corresponding trajectory $\Exp(\lambda,t)$. 
The significance of the Maxwell and conjugate times is the following result:

\begin{theorem}[Theorem~8.72~\cite{AgrachevBarilariBoscain:2020:SRGeomIntro}]
For each $\lambda \in C$ such that the trajectory $\Exp(\lambda,t)$ does not contain abnormal segments, we have
$$\tcut(\lambda) = \min\big(\tconj(\lambda),\tmax(\lambda)\big).$$
\end{theorem}

A common reason for Maxwell points to appear along a trajectory is a symmetry.
\begin{definition}
A pair of mappings 
$$ \Sym \colon N \to N, \qquad \Sym \colon M \to M $$
is called \emph{a symmetry of the exponential mapping} if
$$ \Sym\circ \Exp(\nu) = \Exp \circ \Sym(\nu), \quad \nu \in N,$$
and the first mapping $\Sym$ preserves time.
\end{definition}

\begin{definition}
Let $\Sym$ be a symmetry of the exponential mapping. 
\emph{The Maxwell set corresponding to $\Sym$ in the preimage of the exponential mapping} is 
\begin{align*}
\MAX_{\Sym} = \{\nu \in N \mid \Sym(\nu) \neq \nu, \Exp(\nu) = \Exp \circ \Sym(\nu)\}. 
\end{align*}
The set of \emph{fixed points corresponding to $\Sym$} is
\begin{align*}
\FIX_{\Sym} = \{\nu\in N \mid \Sym(\nu) = \nu\}.
\end{align*}
\end{definition}

A priori the first Maxwell time may not correspond to any symmetry $\Sym$ of $\Exp$.
This happens for instance in the affine on control Euler's elastic problem \cite{Ardentov:2019:HiddenMaxwellElastic}.
However, the first Maxwell time corresponds to a symmetry of $\Exp$ in each of the fully studied left-invariant sub-Riemannian problems on the following groups: the Heisenberg group $\H(3)$~\cite{Vershik_Gershkovich:1987:Nonholonomic_dynamical_systems}; the
groups $\SO(3)$, $\SL(2)$ with axisymmetric metrics~\cite{BoscainRossi:2008:S03SL2}; $\SE(2)$~\cite{Sachkov:2010:CutSE2}; $\SH(2)$~\cite{ButtBhattiSachkov:2017:CutSH2}; the Engel group $\engel$~\cite{ArdentovSachkov:2015:CutEngel}. In this paper, we prove that the same is true for the problem on the Cartan group $\cartan$.

\subsection{Known facts about the Cartan case}\label{subsec-cartan}
The control system~(\ref{sys}) for the left-invariant sub-Riemannian problem on the Cartan group can be specified more explicitly in coordinates as follows:
\begin{align}
\dot{x} & = u_1, \label{dx}\\
\dot{y} & = u_2, \label{dy}\\
\dot{z} & = \frac{- u_1 y + u_2 x}{2}, \label{dz}\\
\dot{v} & = u_2 \frac{x^2 + y^2}{2}, \label{dv}\\
\dot{w} & = -u_1 \frac{x^2 + y^2}{2}, \label{dw} 
\end{align}
where $q=(x,y,z,v,w) \in \cartan \cong \RR^5$.

The family of all normal extremal trajectories of the problem is parametrized by the cylinder
\begin{align*}
C^{\cartan} &= \left\{\lambda \in T_{\qzero}^* \cartan \mid H(\lambda)= 1/2 \right\} \\
 &= \left\{(\theta, c, \alpha, \beta) \in S^1 \times \RR\times\RR \times S^1 \mid \alpha \geq 0  \right\}, 
\end{align*}
where $\alpha, \beta$ represent polar coordinates. The cylinder $C^{\cartan}$ is further decomposed into subsets as follows:
\begin{align*}
&C^{\cartan}=\cup_{i=1}^7 C_i^{\cartan},   \quad  C_i^{\cartan} \cap C_j^{\cartan} = \emptyset, \ i \neq j, \quad \lambda = (\theta,c,\alpha,\beta), \\
&C_1^{\cartan} = \{\lambda \in C^{\cartan} \mid \alpha \neq 0, E\in(- \alpha, \alpha)\}, \\
&C_2^{\cartan} = \{\lambda \in C^{\cartan} \mid \alpha \neq 0, E\in(\alpha,+\infty)\}, \\
&C_3^{\cartan} = \{\lambda \in C^{\cartan} \mid \alpha \neq 0, E=\alpha, c \neq 0 \}, \\
&C_4^{\cartan} = \{\lambda \in C^{\cartan} \mid \alpha \neq 0, E=-\alpha, c = 0, \theta - \beta = 0\}, \\
&C_5^{\cartan} = \{\lambda \in C^{\cartan} \mid \alpha \neq 0, E=\alpha, c = 0, \theta - \beta = \pi\}, \\
&C_6^{\cartan} = \{\lambda \in C^{\cartan} \mid \alpha = 0,  c \neq 0\}, \\
&C_7^{\cartan} = \{\lambda \in C^{\cartan} \mid \alpha = c = 0\},
\end{align*}
where $E  = \frac{c^2}{2} - \alpha \cos (\theta - \beta)$ is the energy of the mathematical pendulum:
\begin{align}
\dot{\theta} = c, \qquad \dot{c} = - \alpha \sin (\theta- \beta), \qquad \dot{\alpha}=\dot{\beta}=0. \label{eq:pendulum}
\end{align}

For each value of the constants $\alpha, \beta$, the pendulum trajectory $(\theta_t, c_t)$ defines a trajectory in $\cartan$ via the exponential mapping
$\Exp^{\cartan} (\lambda,t) = q_t$, where $\lambda = (\theta,c,\alpha,\beta)$ and $(\theta, c):=(\theta_0, c_0) \in S_{\theta}^1 \times \RR_c$ is the initial point of the pendulum trajectory. 

Elliptic coordinates $\lambda = (\varphi, k, \alpha, \beta)$ on the sets $C_1^{\cartan}, C_2^{\cartan}, C_3^{\cartan}$ were intruduced in~\cite{Sachkov:2003:exp_map} for the explicit parametrization of $\Exp^\cartan$. The parameter $\varphi$ is the motion time of the pendulum \eqref{eq:pendulum} from the point of stable equilibrium. The parameter $k$ is a reparametrization of the energy $E$: 
\begin{align*}
&\lambda \in C_1^{\cartan}  &&\Rightarrow && k = \sqrt{\frac{E+\alpha}{2\alpha}} \in (0,1), \\
&\lambda \in C_2^{\cartan}  &&\Rightarrow && k = \sqrt{\frac{2\alpha}{E+\alpha}} \in (0,1), \\
&\lambda \in C_3^{\cartan}  &&\Rightarrow && k = 1. 
\end{align*}

The projections of $\Exp^\cartan(\lambda, t)$ to the plane $(x,y)$ for $\lambda$ in $C_1^{\cartan}$, $C_2^{\cartan}$, and $C_3^{\cartan}$ are inflectional, non-inflectional, and critical Euler elasticae respectively. When $\lambda \in C_{457}^{\cartan} = C_4^{\cartan} \cup C_5^{\cartan} \cup C_7^{\cartan}$ the projections are straight lines parametrized by $\lambda \in S_{\theta}^1$. The projections for $\lambda \in C_6^{\cartan}$ are circles parametrized by $\lambda = (\theta,c), c\neq 0$.

\begin{remark} \label{infinite-finite}
If $\lambda \in C_3^\cartan \cup C_{457}^{\cartan}$, then $\tcut^\cartan(\lambda) = \infty$. 

For $\lambda \in C_1^\cartan \cup C_2^\cartan \cup C_6^\cartan$, the cut time $\tcut^\cartan(\lambda)$ is finite.
\end{remark}

A two-parameter group of continuous symmetries of the exponential mapping is formed by dilations and rotations
\begin{align*}
\delta_\mu &: (\theta, c, \alpha, \beta, t) & &\mapsto  & &(\theta, c/\mu, \alpha/\mu^2, \beta, \mu t), \quad \mu > 0, \nonumber \\
\delta_\mu &: (x, y, z, v, w) & &\mapsto & &(\mu x, \mu y, \mu^2 z, \mu^3 v, \mu^3 w); \\
R_\eta &: (\theta, c, \alpha, \beta, t) & &\mapsto  & &(\theta-\eta, c, \alpha, \beta-\eta, t), \quad \eta \in S^1, \nonumber \\
R_\eta &: (x, y, z, v, w) & &\mapsto & &(x \cos \eta + y \sin \eta, y \cos \eta - x \sin \eta, z, \\
&&&&&  v \cos \eta + w \sin \eta, w \cos \eta - v \sin \eta). \nonumber
\end{align*}

There is also a dihedral group of discrete symmetries $G=\{\Id, \varepsilon^1, \varepsilon^2$,  $\varepsilon^3 = \varepsilon^1 \circ \varepsilon^2\}$ of $\Exp^{\cartan}$, which is described in~\cite{Sachkov:2006:dido_discrete_symmetries}. In terms of the $(x,y)$ projection, the symmetry $\varepsilon^1$ reflects an elastica in the center of its chord; the symmetry $\varepsilon^2$ reflects an elastica in the perpendicular bisector to the chord up to an additional rotation; the symmetry $\varepsilon^3$ reflects an elastica in the chord up to the same additional rotation. Those symmetries generate the corresponding Maxwell sets $\MAX_i:=\MAX_{\varepsilon^i}$ and the corresponding sets of fixed points $\FIX_i:=\FIX_{\varepsilon^i}$ in the preimage of the exponential mapping. For a detailed description of $\MAX_i$, $\FIX_i$, see~\cite{Sachkov:2006:dido_maxwell_set}. Denote the unions of these sets by
\begin{align*}
\MAX = \cup_{i=1}^3 \MAX_i, \qquad \FIX = \cup_{i=1}^3 \FIX_i.
\end{align*}

\begin{lemma}[{\cite[Corollary~2.2, Corollary~2.4, Corollary~3.1, Proposition~3.5]{Sachkov:2006:dido_complete_description_of_maxwell_strata}}]\label{lemma:inclusions}
	\begin{equation*}
	(\lambda,t) \in \MAX \cup \FIX \iff \Exp^{\cartan}(\lambda,t) \in \cartan',
	\end{equation*}
	where
	\begin{align*}
	\cartan'&=\{(x,y,z,v,w)\in\cartan \mid z V=0\},\\
	V &=  x v + y w - \frac{(x^2 + y^2) z}{2}.
	\end{align*}
\end{lemma}

A function $\mathbf{t}\colon C^{\cartan} \to (0, +\infty]$ for the minimal Maxwell time corresponding to the symmetries $\varepsilon^1, \varepsilon^2$ is defined in~\cite{Sachkov:2006:dido_complete_description_of_maxwell_strata}. It gives an upper bound for the first Maxwell time and hence the cut time, i.e.,
\begin{align}
\tcut^\cartan (\lambda)\leq \tmax(\lambda)\leq \mathbf{t}(\lambda), \qquad \forall \lambda \in C^{\cartan}. \label{tbold-above-tcutC}
\end{align}

As mentioned in Remark~\ref{infinite-finite}, $\mathbf{t}(\lambda) = \infty$ for $\lambda \in C_3^\cartan \cup C_{457}^{\cartan}$. Elsewhere 
using dilations $\delta_\mu$, we define the renormalized function $\mu \mathbf{t}(\lambda) = \mathbf{t} \circ \delta_{\mu_\lambda} (\lambda)$ so that one period of the corresponding elastica has unit length. The explicit dilation factors $\mu_\lambda$ and the resulting function $\mu\mathbf{t}$ are
\begin{align*}
&\lambda  \in C_1^{\cartan} & \Rightarrow \quad & \mu_\lambda = \frac{\sqrt{\alpha}}{4 K(k)}, &&\mu \mathbf{t}(\lambda)  = \mathbf{t}_1 (k) = \min\{\mathbf{t}_1^z (k), \mathbf{t}_1^V (k)\}, \\
&\lambda  \in C_2^{\cartan} & \Rightarrow \quad & \mu_\lambda = \frac{\sqrt{\alpha}}{2k K(k)}, &&\mu \mathbf{t}(\lambda) = \mathbf{t}_2 (k) = \mathbf{t}_2^V (k), \\
&\lambda  \in C_6^{\cartan} & \Rightarrow \quad & \mu_\lambda = \frac{|c|}{2 \pi}, &&\mu \mathbf{t}(\lambda) = \mathbf{t}_2 (0) =  \mathbf{t}_2^V (0),
\end{align*}
where $K(k)$ is the complete elliptic integral of the first kind and $\mathbf{t}_1^z (k), \mathbf{t}_1^V (k), \mathbf{t}_2^V (k)$ correspond to the minimal times with vanishing $z$ or $V$. Explicit formulas are given in \cref{appendix-formulas}.

To prove our main theorems, we need the following bounds for $\mathbf{t}_1$ and $\mathbf{t}_2$.

\begin{lemma}\label{lemma:C_1 root bound}
	The Maxwell times $\mathbf{t}_1^z(k), \mathbf{t}_1^V(k)$ satisfy 
	\begin{align*}
&k\in[0,k_0)& &\Rightarrow & &\mathbf{t}_1^z(k)\in(1,3/2), \quad \mathbf{t}_1^V(k) \in (1, 2),\\ 
&k=k_0& &\Rightarrow & &\mathbf{t}_1^z(k)=\mathbf{t}_1^V(k)=1, \\
&k\in(k_0,1)& &\Rightarrow & &\mathbf{t}_1^z(k) \in (1/2, 1), \quad \mathbf{t}_1^V(k) \in (1, 2).
	\end{align*}
The first Maxwell time $\mathbf{t}_1(k)$ satisfies 
\begin{align*}
\mathbf{t}_1(k) =  \begin{cases}
\mathbf{t}_1^z(k), \quad k\in(0,k_1]\cup[k_0,1),\\
\mathbf{t}_1^V(k), \ \quad \quad k\in[k_1,k_0],
\end{cases}
\end{align*}
where $k_0 \approx 0.909$ and $k_1 \approx 0.802$ are roots of certain equations in Jacobi elliptic functions and satisfy $\mathbf{t}_1^z(k)=\mathbf{t}_1^V(k)$.	
\end{lemma}
\begin{proof}
Follows immediately from \cite[Corollary~2.1, Proposition~2.2, Proposition~2.5]{Sachkov:2006:dido_complete_description_of_maxwell_strata}.
\end{proof}

\begin{lemma}\label{lemma:C_2 root bound}
	The first Maxwell time $\mathbf{t}_2(k)$ satisfies 
	\begin{equation*}
	k\in[0,1] \qquad\Rightarrow\qquad \mathbf{t}_2(k) \in [1,2).
	\end{equation*}
\end{lemma}
\begin{proof}
	By \cite[Proposition~3.2]{Sachkov:2006:dido_complete_description_of_maxwell_strata}, we have $\mathbf{t}_2^V(k)\in (1,2)$ for all $k\in(0,1)$. Moreover, by \cite[Proposition~3.4]{Sachkov:2006:dido_complete_description_of_maxwell_strata}, we have also
	\begin{equation*}
	1<\lim\limits_{k\to 0}\mathbf{t}_2^V(k) < 2.
	\end{equation*}
	Finally, by \cref{lemma-f2v-limit-as-k-to-1}, there exists $\tilde{k}<1$ such that $\mathbf{t}_2^V(k)\leq \frac{3}{2}$ for all $\tilde{k}<k<1$. It follows that $\lim\limits_{k\to 1}\mathbf{t}_2^V(k)\leq \frac{3}{2}<2$, proving the statement of the lemma.
\end{proof}

\begin{theorem}[\cite{Sachkov:2021:Cartan_conj_time}]\label{tconjC-above-tbold}
For each $\lambda \in C^{\cartan}$
$$\mathbf{t}(\lambda) \leq \tconj(\lambda).$$
\end{theorem}

By~(\ref{tbold-above-tcutC}) and Theorem~\ref{tconjC-above-tbold}, the subset
$$ \optimalpreimage = \{ (\lambda, t) \in C^\cartan\times[0,\infty) \mid  t \leq \mathbf{t} (\lambda) \} $$
in the preimage of the exponential mapping describes all the potentially optimal geodesics.

In order to prove that all these geodesics are indeed optimal, we study the restriction of $\Exp^{\cartan}$ to the following set:
\begin{align*}
&\preimagedomain:= \widehat{N} \setminus (\FIX \cup \MAX), \\
\Exp(&\preimagedomain) \subset \imagedomain, \qquad \imagedomain = \cartan \setminus \cartan' = \{ (x,y,z,v,w) \in \cartan \mid z V \neq 0 \}.
\end{align*}

\subsection{Comparison with the Engel cut time}
Let us recall the known facts about the solution for the sub-Riemannian problem on the Engel group $\engel$ that we are going to use in the study of $\Exp|_{\preimagedomain}$. 
The control system~(\ref{sys}) for the left-invariant sub-Riemannian problem on the Engel group can be specified by equations (\ref{dx})--(\ref{dv}), i.e., we have a natural projection 
\begin{align*}
\pi_{\engel}\colon \cartan \to \engel, \qquad \pi_{\engel}(x,y,z,v,w)=(x,y,z,v).
\end{align*}

The family of all normal extremal trajectories of the problem is parametrized by the cylinder
\begin{align*}
C^{\engel} &= \left\{\lambda^{\engel} \in T_{\qzero}^* \engel \mid H(\lambda^{\engel})= 1/2 \right\} = \left\{(\theta, c, \alpha) \in S^1 \times \RR\times\RR \right\}.
\end{align*}
We also define a projection between the cylinders
\begin{align*}
\pi_{\engel}: C^{\cartan} \to C^{\engel}, \qquad \pi_{\engel}(\theta, c, \alpha, \beta)=(\theta, c, \alpha).
\end{align*}
The cylinder in the Engel case also has a decomposition $C^{\engel} = \cup_{i=1}^7 C_i^{\engel}$, see~\cite{ArdentovSachkov:2011:ExtremalEngel} for the details. This decomposition satisfies
\begin{align*}
C_i^{\engel}\cap \{\alpha \geq 0 \} &= \pi_{\engel} \big(C_i^{\cartan} \cap \{\beta = 0\}\big), \quad i = 1 \dots 7. 
\end{align*}
The case $\alpha<0$ is symmetric to the case $\alpha > 0$.

\begin{lemma}
Let $q_t=\Exp^{\cartan}(\lambda^{\cartan},t), t \geq 0$ be an extremal trajectory for the sub-Riemannian problem on $\cartan$ with $\lambda^{\cartan}=(\theta,c,\alpha,\beta) \in C^{\cartan}$. Then $\Exp^{\engel} (\bpiE (\lambda^{\cartan}), t) = \bpiE (q_t)$ is an extremal trajectory for the sub-Riemannian problem on $\engel$, where 
$$\bpiE = \pi_{\engel} \circ R_{\beta-\pi/2}.$$
\end{lemma}
\begin{proof}
	Follows directly from the explicit formulas for the exponential mappings in the Cartan case \cite{Sachkov:2003:exp_map} and in the Engel case \cite{ArdentovSachkov:2011:ExtremalEngel}.
\end{proof}

\begin{theorem}[\cite{ArdentovSachkov:2015:CutEngel}]\label{thm:engel-cut}
For each covector $\lambda\in C^\engel$, let $\mu_\lambda$ be the dilation factor such that the corresponding elastica has unit length period.
The normalized function $\mu \tcut^{\engel}: C^{\engel} \to (0, +\infty]$ for the cut times in the Engel group has the following form:
\begin{align*}
&\lambda \in C_1^{\engel} &&\Rightarrow & \mu_\lambda &= \frac{\sqrt{\abs{\alpha}}}{4 K(k)}, &\mu \tcut^{\engel} (\lambda) &= \min\{1,\mathbf{t}_1^z (k)\}, \\
&\lambda \in C_2^{\engel} &&\Rightarrow & \mu_\lambda &= \frac{\sqrt{\abs{\alpha}}}{2 k K(k)}, &\mu \tcut^{\engel} (\lambda) &= 1, \\
&\lambda \in C_6^{\engel} &&\Rightarrow &\mu_\lambda &= \frac{|c|}{2 \pi}, &\mu \tcut^{\engel} (\lambda) &= 1.
\end{align*}
\end{theorem}

If the projection of a Cartan extremal trajectory is optimal in the Engel group, then the Cartan trajectory must be optimal as well. Hence we have the inequality
\begin{align}
\tcut^{\cartan} (\lambda^{\cartan}) \geq  \tcut^{\engel} \circ \bpiE (\lambda^{\cartan}), \qquad \forall \lambda^{\cartan} \in C^{\cartan}. \label{tcutC-above-tcutE}
\end{align}

As a consequence of Lemmas \ref{lemma:C_1 root bound}--\ref{lemma:C_2 root bound}, we can bound the conjectured cut times in the Cartan group by the corresponding cut times in the Engel group:
\begin{lemma}\label{lemma:engel cut time comparison}
	There exists a constant $1\leq\zeta<2$ such that
	for every $\lambda^\cartan \in C_1^\cartan\cup C_2^\cartan$, we have
	\begin{equation*}
	\mathbf{t}(\lambda^\cartan)\leq \zeta\cdot\tcut^{\engel}(\lambda^\engel),
	\end{equation*}
	where $\lambda^\engel:= \bpiE (\lambda^\cartan)$.
\end{lemma}
\begin{proof}
	Since $\bpiE(C_i^\cartan)\subset C_i^\engel$, we find that
	\begin{equation*}
	\frac{\mathbf{t}(\lambda^\cartan)}{\tcut^{\engel}(\lambda^\engel)} = \begin{dcases}
	\frac{\min\big\{\mathbf{t}_1^z (k),\mathbf{t}_1^V(k)\big\}}{\min\big\{\mathbf{t}_1^z (k),1\big\}},&\lambda\in C_1^\cartan,\\
	\mathbf{t}_2(k),&\lambda\in C_2^\cartan.
	\end{dcases}
	\end{equation*}
	Therefore, if we set 
	\begin{equation*}
	\zeta=\max\Big\{ \max_{k\in[0,1]} \mathbf{t}_1^z(k), \max_{k\in[0,1]} \mathbf{t}_2(k) \Big\},
	\end{equation*}
	the inequality $\mathbf{t}(\lambda^\cartan)\leq \zeta\cdot\tcut^{\engel}(\lambda^\engel)$ follows, so it remains to show that $\zeta<2$.
	
	This bound follows from the earlier estimates on $\mathbf{t}_1$ and $\mathbf{t}_2$. Namely, we get the bounds $\max_{k\in[0,1]} \mathbf{t}_1^z(k)\leq \frac{3}{2}$ and $\max_{k\in[0,1]} \mathbf{t}_2(k)< 2$ from \cref{lemma:C_1 root bound} and \cref{lemma:C_2 root bound}.
\end{proof}

\section{Properness of the sub-Riemannian exponential map}\label{section-properness}

\begin{definition}
	A map $f\colon X\to Y$ 
	is \emph{proper} if $f^{-1}(K)\subset X$ is compact for any compact set $K\subset Y$.
\end{definition}

The goal of this section is to prove the following result.

\begin{proposition}\label{prop-exp-is-proper}
	The restriction $\Exp^{\cartan}\colon \preimagedomain\to \imagedomain$ of the sub-Riemannian exponential is a proper map.
\end{proposition}

For the proof of properness, the following notion is convenient.

\begin{definition}
	Let $X$ be a topological space. A sequence $(x_j)_{j\in\NN} \in X$ is said to be \emph{escaping} if it eventually exits any compact set. That is, for any compact set $K\subset X$, there exists $j_0\in\NN$ such that for $j\geq j_0$, we have $x_j\in X\setminus K$.
\end{definition}

Recall that in metric spaces properness is characterized by preserving escaping sequences. That is, if $f\colon X\to Y$ is a continuous map between metric spaces $X$ and $Y$, then $f$ is proper if and only if $f(x_j)_{j\in\NN}\subset Y$ is escaping for every escaping sequence $(x_j)_{j\in\NN}\subset X$.

\begin{remark}
	When referring to escaping sequences, we use the notation $(\lambda_n,t_n)\to\partial \preimagedomain$ and $\Exp^{\cartan}(\lambda_n,t_n)\to\partial \imagedomain$. The boundaries $\partial \preimagedomain$ and $\partial \imagedomain$ are understood inside the one-point compactifications of $C^\cartan\times [0,\infty)$ and $\cartan$ respectively, in order to also handle the case when $(\lambda_n,t_n)\to\infty$.
\end{remark}

The proof of \cref{prop-exp-is-proper} is given in \cref{section-proof-of-properness} by considering two types of escaping sequences $(\lambda_n,t_n)\to\partial\preimagedomain$. 

The first case is when the sequence $(t_n)_{n\in\NN}$ stays bounded. Then the claim that $\Exp(\lambda_n,t_n)\to\partial\imagedomain$ will follow by continuity by considering a (possibly abnormal) limit of the corresponding trajectories $t\mapsto \Exp(\lambda_n,t)$.

The second case is when $(t_n)_{n\in\NN}$ is instead unbounded. Then the proof is more involved, and follows by a comparison with the known cut times in the Engel case. To make this comparison easier, we consider two simplifications in \cref{section-reduction-to-C1-and-C2}. First, we reduce to the dense subset of the points $(\lambda,t)\in \preimagedomain$ with $\lambda\in C_1^{\cartan}\cup C_2^{\cartan}$. Second, using the rotational symmetry of the sub-Riemannian exponential map, we further reduce to $\lambda=(\theta,c,\alpha,\beta)\in C_1^{\cartan}\cup C_2^{\cartan}$ with $\beta=0$.

\subsection{Reduction to rotated generic elasticae}\label{section-reduction-to-C1-and-C2}
A priori we have to consider escaping sequences $(\lambda_n,t_n)\to \partial \preimagedomain$ with arbitrary $\lambda_n\in C^{\cartan}$. However, since $C_1^{\cartan}\cup C_2^{\cartan}$ is dense in $C^{\cartan}=C_1^{\cartan}\cup\dots\cup C_7^{\cartan}$, such sequences are well approximated by escaping sequences with $\lambda_n\in C_1^{\cartan}\cup C_2^{\cartan}$. More precisely, we have the following lemma.

\begin{lemma}\label{lemma:proper on dense subset}
	Suppose $X$ and $Y$ are boundedly compact metric spaces, $f\colon X\to Y$ is a continuous map, and $U\subset X$ is a dense subset.
	Assume that if $(u_j)_{j\in\NN}\subset X$ is an escaping sequence and $u_j\in U$ for all $j\in\NN$, then the sequence $f(u_j)_{j\in\NN}\subset Y$ is also escaping. Then $f$ is proper.
\end{lemma}
\begin{proof}
	Let $(x_j)_{j\in\NN}\subset X$ be any escaping sequence. We need to verify that $f(x_j)_{j\in\NN}\subset Y$ is also escaping.
	
	By continuity of the function $f$ and denseness of the subset $U\subset X$, there exist points $u_j\in U$ such that
	\begin{equation}\label{eq:defn of approx seq}
	d_X(x_j,u_j)<1/j\quad\text{and}\quad
	d_Y(f(x_j),f(u_j))<1/j\quad\forall j\in\NN.
	\end{equation}
	If $K\subset X$ is any compact set, then by bounded compactness of $X$, also the set 
	\begin{equation*}
	\overline{B}(K,1) = \{x\in X \mid d(x,K)\leq 1\}
	\end{equation*}
	is compact. If for some $j\in\NN$, we have $x_j\notin \overline{B}(K,1)$, then $u_j\notin K$ by \eqref{eq:defn of approx seq}. Therefore the assumption that $(x_j)_{j\in\NN}\subset X$ is escaping implies that $(u_j)_{j\in\NN}\subset X$ is escaping.
	
	By the assumption of the lemma, the sequence $f(u_j)_{j\in\NN}$ is escaping. Arguing exactly as before with the role of $x_j$ and $u_j$ taken by $f(u_j)$ and $f(x_j)$, we see that bounded compactness of $Y$ and \eqref{eq:defn of approx seq} imply that $f(x_j)_{j\in\NN}$ is escaping.
\end{proof}

\begin{lemma}\label{lemma:reduction to beta=0}
	Suppose that $\Exp^{\cartan}(\theta_n,c_n,\alpha_n,0,t_n)\to\partial \imagedomain$ for any escaping sequence $(\theta_n,c_n,\alpha_n,0,t_n)\to \partial \preimagedomain$ in $\preimagedomain$.
	Then $\Exp^{\cartan}(\lambda_n,t_n)\to\partial \imagedomain$ for any escaping sequence $(\lambda_n,t_n)\to \partial \preimagedomain$ in $\preimagedomain$.
\end{lemma}
\begin{proof}
	If $(\lambda_n,t_n)=(\theta_n,c_n,\alpha_n,\beta_n,t_n)\to \partial \preimagedomain$ is an escaping sequence, so is the rotated sequence $R_{\beta_n}(\lambda_n,t_n)_{n\in\NN}$. Since $R_{\beta_n}(\lambda_n,t_n)=(\theta_n-\beta_n,c_n,\alpha_n,0,t_n)$, the assumption of the lemma implies that $\Exp^{\cartan}(\theta_n-\beta_n,c_n,\alpha_n,0,t_n)\to\partial \imagedomain$. 
	
	Rotations preserve both the coordinates $z$ and $V$, so the set $\imagedomain$ and its boundary $\partial\imagedomain$ are invariant under the rotations.
	It follows that
	\begin{equation*}
	\Exp^{\cartan}(\lambda_n,t_n) = R_{-\beta_n}\circ \Exp^{\cartan}(\theta_n-\beta_n,c_n,\alpha_n,0,t_n) \to \partial\imagedomain.
	\end{equation*}
\end{proof}

\subsection{Proof of properness}\label{section-proof-of-properness}
We will next conclude the proof of \cref{prop-exp-is-proper} that the restriction $\Exp^{\cartan}\colon\preimagedomain\to\imagedomain$ of the sub-Riemannian exponential in the Cartan group is proper.
The proof splits into two cases based on boundedness of $(t_n)_{n\in\NN}$.

In the bounded time case, we show that the sequence is escaping by considering a limiting trajectory of the extremal trajectories $t\mapsto \Exp(\lambda_n,t)$.

\begin{lemma}\label{lemma:abnormal limit}
	Suppose $(\lambda_n,t_n)_{n\in\NN}\subset \preimagedomain$ is an escaping sequence with $(t_n)_{n\in\NN}$ bounded. Then any limit point of the sequence $\Exp^{\cartan}(\lambda_n,t_n)$ is contained in $\cartan'=\{zV=0\}$.
\end{lemma}
\begin{proof}
	
	Fix $T>0$ such that $t_n\in [-T,T]$ for all $n\in \NN$. Consider the family of normal trajectories $q_n\colon [-T,T]\to \cartan$, $q_n(t)=\Exp^{\cartan}(\lambda_n,t)$ with controls $\mathbf{u}_n$. By construction, these satisfy the PMP for the normal covector pair $(-1,\lambda_n)\in (-\infty,0]\times T_{\Id}^*\cartan$. That is, for all controls $\mathbf{u}$, we have
	\begin{equation}\label{eq:pmp}
	\lambda_n\left( \int_{-T}^{T}\Ad_{q_n(t)}\mathbf{u}(t)\,dt \right) - \left<\mathbf{u}_n,\mathbf{u}\right> = 0.
	\end{equation}
	
	Let $\qone\in \cartan$ be a limit point of the sequence of points $\Exp^{\cartan}(\lambda_n,t_n)\in \cartan$. Up to taking a subsequence, we may assume that $\Exp^{\cartan}(\lambda_n,t_n)\to \qone$. Since the trajectories $q_n$ are all 1-Lipschitz curves through $q_n(0)=\identity$, up to taking a further subsequence, we may assume by Arzelà-Ascoli that there exists a limit trajectory $q_\infty\colon[-T,T]\to \cartan$ such that $q_n\to q_\infty$ uniformly.
	
	If the sequence $(\lambda_n)_{n\in\NN}$ of covectors is bounded in $C^\cartan$, there exists a limit point $(\lambda_n,t_n)\to(\bar{\lambda},\bar{t})\in \optimalpreimage$. 
	The assumption that $(\lambda_n,t_n)_{n\in\NN}$ is escaping in $\preimagedomain$ implies that $(\bar{\lambda},\bar{t})\in \optimalpreimage\setminus \preimagedomain = \MAX \cup \FIX$.
	By Lemma~\ref{lemma:inclusions}, $\qone=\Exp(\bar{\lambda},\bar{t})\in \cartan'$.
	
	Suppose instead that the sequence $(\lambda_n)_{n\in\NN}$ of covectors is unbounded in $C^\cartan$. Let $(a_n)_{n\in\NN}\subset(0,\infty)$ be a sequence such that there exists a finite non-zero limit
	$\lambda_\infty:=\lim\limits_{n\to\infty}a_n\lambda_n\in T_{\Id}^*\cartan$. The assumption that $(\lambda_n)_{n\in\NN}$ is unbounded implies that necessarily $a_n\to 0$.
	
	Rescaling \eqref{eq:pmp} by the factors $a_n\in(0,\infty)$, each trajectory $q_n$ satisfies the PMP for the normal pair $(-a_n,a_n\lambda_n)\in(-\infty,0]\times T_{\Id}^*\cartan$. That is, for all controls $\mathbf{u}$, we have
	\begin{equation*}
	a_n\lambda_n\left( \int_{-T}^{T}\Ad_{q_n(t)}\mathbf{u}(t)\,dt \right) - a_n\left<\mathbf{u}_n,\mathbf{u}\right> = 0.
	\end{equation*}
	By continuity, we conclude that the limit trajectory $q_\infty$ satisfies the PMP for the abnormal pair $(0,\lambda_\infty)\in(-\infty,0]\times T_{\Id}^*\cartan$.
	
	Since the only abnormal curves are horizontal lines, we see that $q_\infty$ is contained in $\{z=V=0\}$. Finally, since $t_n\in [-T,T]$ for all $n\in \NN$, by uniform convergence, we conclude that
	\begin{equation*}
	\qone = \lim\limits_{n\to\infty}q_n(t_n)
	= q_\infty(\lim\limits_{n\to\infty}t_n),
	\end{equation*}
	so the limit point $\qone$ is contained in $\cartan'$.
\end{proof}

In the unbounded time case, we show that the sequence is escaping by comparing to the distance on the Engel group.

\begin{lemma}\label{lemma-unbdd-time}
	Let $(\lambda_n^{\cartan},t_n)_{n\in\NN}\subset \preimagedomain$ be a sequence with $t_n\to\infty$. Then $d(\identity,\Exp^{\cartan}(\lambda_n^{\cartan},t_n))_{n\in\NN}\to\infty$.
\end{lemma}
\begin{proof}
	By \cref{lemma:engel cut time comparison}, there exists a constant $\zeta\in[1,2)$ independent of $\lambda_n^{\cartan}$ such that $\mathbf{t}(\lambda_n^{\cartan}) \leq \zeta\cdot \tcut^{\engel}(\lambda_n^{\engel})$, where $\lambda_n^{\engel} = \bpiE (\lambda_n^{\cartan})$. By assumption $(\lambda_n^{\cartan},t_n)\in \preimagedomain$, so 
	\begin{equation*}
	t_n<\mathbf{t}(\lambda_n^{\cartan})\leq \zeta\cdot\tcut^{\engel}(\lambda_n^{\engel}).
	\end{equation*}
	Denote the trajectories by $q_n(t) := \Exp^{\cartan}(\lambda_n^{\cartan},t)$ for short.
	
	If $\tcut^{\engel}(\lambda_n^{\engel})<t_n\leq \zeta\cdot \tcut^{\engel}(\lambda_n^{\engel})$, then, by the triangle inequality, we can bound
	\begin{equation*}
	d\big(\identity,q_n(t_n)\big) \geq d\Big(\identity,q_n\big(\tcut^{\engel}(\lambda_n^{\engel})\big)\Big) - d\Big(q_n\big(\tcut^{\engel}(\lambda_n^{\engel})\big),q_n(t_n)\Big).
	\end{equation*}
	Since the trajectories $q_n(t)$ are 1-Lipschitz and optimal on the interval $[0,\tcut^{\engel}(\lambda_n^{\engel})]$, the above can be further estimated by
	\begin{equation*}
	d\big(\identity,q_n(t_n)\big) \geq (2-\zeta)\,\tcut^{\engel}(\lambda_n^{\engel}) \geq \frac{(2-\zeta)}{\zeta}t_n.
	\end{equation*}
	On the other hand, if $t_n\leq \tcut^{\engel}(\lambda_n^{\engel})$, then the trajectory is already optimal in the Engel group, and hence is also optimal in $\cartan$. That is, we have
	\begin{equation*}
	d(\identity,q_n(t_n))= t_n.
	\end{equation*}
	In either case, the assumption that $t_n\to\infty$ implies that $d\big(\identity,q_n(t_n)\big)\to\infty$ as $n\to\infty$.
\end{proof}

Up to taking subsequences, \cref{prop-exp-is-proper} follows by combining \cref{lemma:abnormal limit} and \cref{lemma-unbdd-time}.

\section{Cut time} \label{section-cut-time}

We now have all the ingredients to verify the conjectured cut times of \cite{Sachkov:2006:dido_complete_description_of_maxwell_strata}.

\subsection{Proofs of the main theorems}

The first result is the uniqueness of geodesics for the points $\imagedomain = \{(x,y,z,v,w)\in\cartan \mid zV\neq 0\}$, where $V = xv-yw-(x^2+y^2)z/2$.

\begin{theorem}\label{thm:exp-diffeo}
$\Exp^\cartan \colon \preimagedomain \to \imagedomain$ is diffeomorphism.
\end{theorem}
\begin{proof}
The group $S^1$ of rotations $R_\eta$ acts freely on both $\preimagedomain$ and $\imagedomain$. Since the rotations are symmetries of the sub-Riemannian exponential $\Exp^\cartan \colon \preimagedomain \to \imagedomain$, it follows that the exponential descends to a well defined smooth map $\quotientexp \colon \preimagedomain /S^1\to\imagedomain/S^1$.
Since the action on $\imagedomain$ is free, it suffices to prove that the quotient map $\quotientexp$ is a diffeomorphism.

First, by \cref{prop-exp-is-proper}, the exponential $\Exp^\cartan \colon \preimagedomain \to \imagedomain$ is proper, so $\quotientexp$ is proper as well.
Second, by Theorem~\ref{tconjC-above-tbold},
the Jacobian of $\Exp^\cartan$ is nonzero everywhere in $\preimagedomain$, so also the Jacobian of $\quotientexp$ is everywhere nonzero.
The claim then follows by the Hadamard Global Diffeomorphism Theorem \cite[Theorem~6.2.8]{KrantzParks:2002:implicit_function_theorem} once we show that the connected components of $\imagedomain/S^1$ are simply connected.

The connected components of $\imagedomain/S^1$ are all homeomorphic to the subset
\begin{equation*}
\imagedomain_{0+} = \{(x,y,z,v,w)\in\imagedomain\mid x=0,y>0,z>0,V>0 \},
\end{equation*}
which further homotopy retracts to the level set 
\begin{equation*}
\imagedomain_{02}=\{(x,y,z,v,w)\in \imagedomain_{0+}\mid z=2\}. 
\end{equation*}
For any $(x,y,z,v,w)\in \imagedomain_{02}$, the function $V$ has the simplified expression
\begin{equation*}
V(x,y,z,v,w) = y(w-y).
\end{equation*}
That is, for points in $\imagedomain_{02}$, we have $V>0$ if and only if $w>y$. Hence $\imagedomain_{02}$ is convex and, in particular, simply connected, so the same is true for the connected components of $\imagedomain/S^1$, concluding the proof of the theorem.
\end{proof}

\cref{thm:exp-diffeo} gives the following coordinate version of \cref{theorem-unique-minimizer-region}.

\begin{theorem}\label{theorem-unique-minimizer-region-coords}
If $\qone=(x,y,z,v,w)\in\cartan$ is such that $z\neq 0$ and $x v + y w - \frac{(x^2 + y^2) z}{2}\neq 0$, then there exists a unique minimizer from $\Id$ to $\qone$.
\end{theorem}

Using \cref{thm:exp-diffeo}, we confirm the conjectured cut times of \cite{Sachkov:2006:dido_complete_description_of_maxwell_strata}, proving also \cref{theorem-cuttime-is-maxwell-time}.

\begin{theorem}\label{tcut-is-tbold}
For $\lambda \in C^\cartan$, we have
\begin{align}
\tcut^\cartan (\lambda) = \mathbf{t} (\lambda).
\end{align}
\end{theorem}
\begin{proof}
The case $\lambda \in C_3^\cartan \cup C_{457}^\cartan$ when $\mathbf{t} (\lambda) = + \infty$ follows by equation~(\ref{tcutC-above-tcutE}). 

If $\lambda \in C_1^\cartan \cup C_2^\cartan$, then we have a finite $\mathbf{t}(\lambda) \in (0, \infty)$ with $(\lambda, \mathbf{t}(\lambda)) \in \MAX \cup \FIX$. Points $(\lambda,t)\in\FIX$ are described by the equation $\sn \tau \cn \tau = 0$ \cite{Sachkov:2006:dido_complete_description_of_maxwell_strata}, where $\tau$ is given by
\begin{align*}
&\lambda \in C_1^{\cartan} && \Rightarrow && \tau = \sqrt{\alpha} \Big(\varphi + \frac{t}{2}\Big), \\
&\lambda \in C_2^{\cartan} && \Rightarrow && \tau = \frac{\sqrt{\alpha}}{k} \Big(\varphi + \frac{t}{2}\Big).
\end{align*}
Since the zeros of the equation are isolated with respect to $t$, there exists $\epsilon>0$, s.t. $(\lambda, t) \notin \MAX \cup \FIX$ for all $t\in\big(\mathbf{t}(\lambda)-\epsilon, \mathbf{t}(\lambda)\big)$. 

Therefore, for such $t$, we have $q=\Exp^\cartan (\lambda,t) \in \imagedomain$.
By \cref{lemma:inclusions} and \cref{thm:exp-diffeo}, $\nu=(\lambda, t)$ is the unique solution to $\Exp(\nu) = q$ among all the potentially optimal geodesics $\nu\in \optimalpreimage$.
By continuity, $(\lambda,\mathbf{t}(\lambda))$ is optimal. 

The case $\lambda \in C_6^{\cartan}$ follows from the case $\lambda \in C_2^{\cartan}$, since $\mathbf{t}(\lambda)$ is continuous for $\lambda\in C_2^{\cartan}\cup C_6^{\cartan}$ by \cite[Proposition 3.4]{Sachkov:2006:dido_complete_description_of_maxwell_strata}.
\end{proof}

\begin{proof}[Proof of \cref{theorem-cuttime-comparison}]
	By \cref{tcut-is-tbold}, $\tcut^\cartan=\mathbf{t}$.
	Hence by \cref{lemma:engel cut time comparison}, the inequality
	\begin{equation}\label{eq-cuttime-inequality}
	\tcut^{\engel}(\lambda^\engel)\leq \tcutc(\lambda^\cartan) \leq \zeta\cdot \tcute(\lambda^\engel)
	\end{equation}
	holds for all $\lambda^\cartan\in C_1^\cartan\cup C_2^{\cartan}$, where $\lambda^\engel= \bpiE(\lambda^\cartan)$.
	
	By continuity of $\mathbf{t}$ on $C_1^\cartan\cup C_2^{\cartan}\cup C^\cartan_6$, inequality \eqref{eq-cuttime-inequality} extends also to $\lambda^\cartan\in C_6^\cartan$. On the other hand, by \cref{infinite-finite}, $\tcutc(\lambda^\cartan) = \tcut^{\engel}(\lambda^\engel) = \infty$ elsewhere, so the claim holds for all $\lambda^\cartan\in C^\cartan$.
\end{proof}

\subsection{Properties of the Cartan and Engel cut times}

\begin{corollary}\label{Cartan-invariant-properties}
	The function $\tcut^\cartan\colon C^\cartan \to (0,+\infty]$ has the following properties:
	\begin{itemize}
		\item $\tcut^\cartan(\lambda)$ depends only on the Casimirs $E$, $\alpha$. In particular, the Cartan group is equioptimal.
		\item $\tcut^\cartan(\lambda)$ is homogeneous with respect to dilations: 
		\begin{equation*}
		\tcut^\cartan\circ \delta_\mu (\lambda) = \mu\cdot \tcut^\cartan(\lambda)\quad\text{for all }\lambda\in C^\cartan.
		\end{equation*}
	\end{itemize}
\end{corollary}
\begin{remark}
	The function $\tcut^\engel(\lambda)$ has the same properties.
\end{remark}

Now consider the functions $\tcut^\cartan(\lambda^\cartan)$, $\tcut^\engel(\lambda^\engel)$ for the most general cases when $\lambda^\cartan = (\varphi, k, \alpha, \beta) \in C_i^\cartan, \lambda^\engel = (\varphi, k, \alpha) \in C_i^\engel, i = 1,2$. Both functions depend only on the parameters $k, \alpha$, where $k$ determines the shape of the elastica on the plane $(x,y)$ and the parameter $\alpha$ changes the size of the elastica. Normalizing the full period of the elastica to unit length by the dilations $\delta_{\mu_\lambda}$, we compare the corresponding cut times $\mu \tcut^{\engel} (k), \mu \tcut^{\cartan} (k)$ for the problems on the Engel and Cartan groups, see Fig.~\ref{tcutC12}. The corresponding optimal elasticae are shown in \cref{fig-optimal-family}.

\begin{remark}
Numerical calculations show that the optimal bounds for the constant $\zeta$ in \cref{lemma:engel cut time comparison} are given by
	\begin{align*}
	\lambda  &\in C_1^{\cartan} & &\Rightarrow& \tcut^{\cartan}(\lambda^\cartan)&\leq \mathbf{t}_1^z (0) \cdot\tcut^{\engel}(\lambda^\engel)\\
	\lambda  &\in C_2^{\cartan}\cup  C_6^\cartan & &\Rightarrow& \tcut^{\cartan}(\lambda^\cartan)&\leq \mathbf{t}_2^V (0) \cdot\tcut^{\engel}(\lambda^\engel),
	\end{align*}
with $\lambda^\engel:= \bpiE (\lambda^\cartan)$, where $\mathbf{t}_1^z (0) \approx 1.430$ is the first positive root $t=\mathbf{t}_1^z (0)$ of $\sin (\pi t) = \pi t \cos(\pi t)$; 
the value $\mathbf{t}_2^V (0) \approx 1.465$ is the first positive root $t=\mathbf{t}_2^V (0)$ of
\begin{equation*}
\cos (\pi t) \big(2 \pi^2 t^2 - \sin^2(\pi t)\big) = \pi t \sin (\pi t).
\end{equation*}
Note that $\mathbf{t}_2^V(0)$ is the cut time for the Cartan geodesics projecting to the circle with unit circumference.
\end{remark}

\begin{figure}[htbp]
\includegraphics[width=0.49\textwidth]{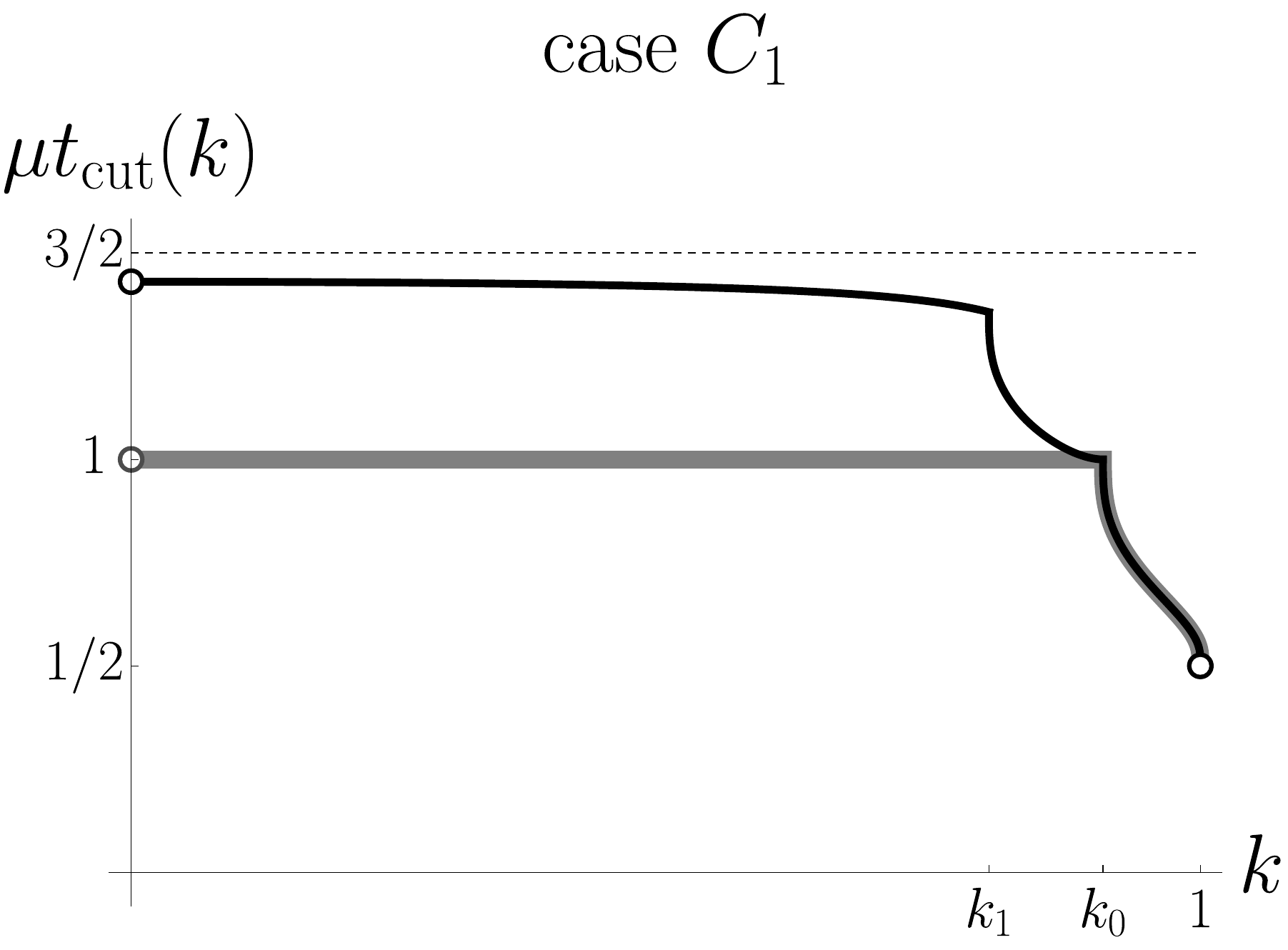}
\hfill
\includegraphics[width=0.49\textwidth]{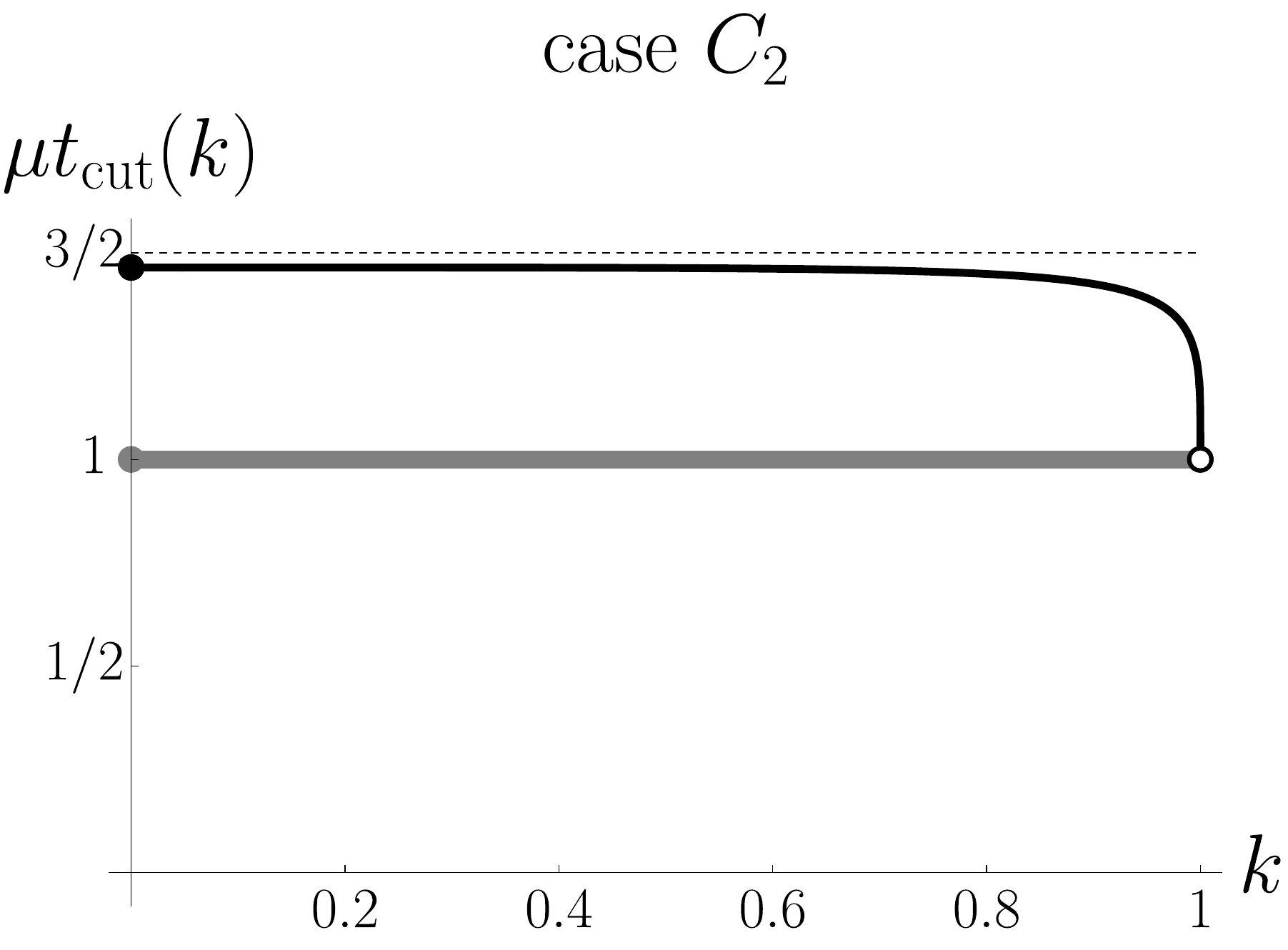}
\\
\caption{Normalized cut times $\tcut^\engel(k)$ (gray), $\tcut^\cartan(k)$ (black) for general types of elasticae: inflectional (left); non-inflectional (right) with circle when $k=0$}
\label{tcutC12}
\end{figure}
\section{Open questions} \label{section-open-questions}
Our work opens three immediate avenues of further research.

Our study reduces the boundary problem~(\ref{sys})--(\ref{integralSR}) in the general situation of the Cartan case when $\qone \in \imagedomain$ to finding the unique root $(\lambda, t) \in \preimagedomain$ of the five-dimensional system of equations $\Exp^\cartan(\lambda, t) = \qone$. Using the continuous symmetries, it is possible to reduce the number of equations of the system to three. Software for solving a similar three-dimensional system of equations is described in~\cite{MashtakovArdentovSachkov:2013:Software}.
By means of nilpotentization, such a software is useful for approximate solving of sub-Riemannian problems with growth vector $(2,3,5)$. An iterative algorithm based on nilpotent approximation was developed in~\cite{Mashtakov:2012:Algorithms} to find the approximate solution of a generic (2,3,5)-problem and applied for two such problems: the plate-ball problem and suboptimal control of a wheeled robot with two passive off-hooked trailers. 
See also \cite{Ardentov:2016:ControllingNilpotent} and \cite{ArdentovMashtakov:2021:ControlNilpotent} for suboptimal control of a robot with a single trailer via nilpotent approximation with the sub-Riemannian Engel group.

We can also conclude that the cut locus in the sub-Riemannian problem on the Cartan group $\Cut \subset \cartan$ lies in the domain of fixed points of the symmetries $\varepsilon^1, \varepsilon^2$, i.e., $\Cut \subset \cartan'$. However, a complete description of the cut locus and the multiplicity of the solutions for $\qone\in \cartan'$ remains unknown and requires a separate investigation. 

The study of the corresponding sub-Riemannian spheres and their singularities are of interest to specialists in various fields of mathematics. Numerical evidence allows us to suggest that the sub-Riemannian distance and the spheres are not subanalytic in the Cartan case similarly to the flat Martinet and Engel cases \cite{AgrachevBonnardChybaKupka:1997:martinet,ArdentovSachkov:2015:CutEngel}.\appendix
\section{Formulas for the first Maxwell times}\label{appendix-formulas}
Here we give the relevant formulas from~\cite{Sachkov:2006:dido_complete_description_of_maxwell_strata} used in~\cref{section-known-results}. Additionally, we formulate and prove the technical~\cref{lemma-f2v-limit-as-k-to-1} required for the proof of properness.

In the case $\lambda \in C_1^\cartan$, the first Maxwell times $\mathbf{t}_1^z(k) = \frac{p_1^z(k)}{2 K(k)}$, $\mathbf{t}_1^V(k) = \frac{p_1^V(k)}{2 K(k)}$ corresponding to the symmetries $\varepsilon^1, \varepsilon^2$ respectively are defined by the first positive roots $p=p_1^z(k), p=p_1^V(k)$ of the equations $f_1^z(p,k)=0$,  $f_1^V(p,k)=0$,  where
\begin{align*}
f_1^z(p,k) =&\sn p \dn p - g_1(p) \cn p, \qquad g_1(p) = 2 \E(p) - p,\\
f_1^V(p,k) =& \frac{4}{3} \dn p \sn p \Big(g_1^3(p) - p - 2 g_1(p) \big(1 - (2 - 6 \cn^2 p) k^2\big)  \\
&+ 8 k^2 \sn p \cn p \dn p\Big) + 4 \cn p\; g_1^2(p) (1 - 2 k^2 \sn^2 p),
\end{align*}
$\sn$, $\cn$, $\dn$ are Jacobi elliptic functions; $\E(p)$ is the composition of the incomplete elliptic integral of the second kind with the elliptic amplitude (the inverse function to the incomplete elliptic integral of the first kind). We do not write the second parameter (elliptic modulus) for short, since it always coincides with $k$ for every function. 

We define the function $\mathbf{t}_1^z(k)$ at $k=0$ as 
\begin{equation*}
\mathbf{t}_1^z(0)=\lim_{k\to +0} \mathbf{t}_1^z(k) = \frac{p_1^z(0)}{\pi}, 
\end{equation*}
where $p=p_1^z(0)$ is the first positive root of the equation
\begin{equation*}
f_1^z(p,0) = \sin p - p \cos p = 0.
\end{equation*}

In the case $\lambda \in C_2^\cartan$, the first Maxwell time $\mathbf{t}_2^V(k) = \frac{p_2^V(k)}{K(k)}$ corresponding to the symmetry $\varepsilon^2$ is defined by the first positive root $p=p_2^V(k)$ of the equation $f_2^V(p,k)=0$, where
\begin{align}
f_2^V(p,k) =& \frac{4}{3} \bigg(\dn p \big(8 k^2 \cn^2 p  \sn^2 p + g_2^2(p) (3 - 6 \sn^2 p)\big) + \cn p \sn p  \nonumber \\
&\times \Big(g_2^3(p) - k^4 p - 2 g_2(p) \big(4 + k^2 (1 - 6 \sn^2 p)\big)\Big)\bigg), \label{eq:f_V in C_2}\\
g_2(p) =& 2\E(p)-(2-k^2)p. \nonumber
\end{align}

In the case $\lambda \in C_6^\cartan$, the first Maxwell time $\mathbf{t}_2^V(0) = \frac{2 p_2^V(0)}{\pi}$ corresponding to the symmetry $\varepsilon^2$ is defined by the first positive root $p=p_2^V(0)$ of equation $f_2^V(p,0)=0$, where
\begin{align*}
f_2^V(p,0) =& \frac{1}{512}\big((32 p^2 - 1) \cos(2 p) - 8 p \sin(2 p) + \cos(6 p)\big).
\end{align*}

\begin{lemma}\label{lemma-f2v-limit-as-k-to-1}
	There exists a value $\tilde{k}<1$ such that $f_2^V(p,k)$ has a root in the interval $p\in (K(k),\frac{3}{2}K(k)]$ for all $\tilde{k}<k<1$. In particular, $\mathbf{t}_2^V(k)\leq \frac{3}{2}$ for all $\tilde{k}<k<1$.
\end{lemma}
\begin{proof}
	In \cite[Equation~(18)]{Sachkov:2006:dido_complete_description_of_maxwell_strata}, it is shown that $f_2^V(K(k),k)<0$ for all $0<k<1$. Therefore, it suffices to show that there exists some $\tilde{k}<1$ such that $f_2^V(\frac{3}{2}K(k),k)\geq 0$ for all $k>\tilde{k}$.
	
	We consider the asymptotics of expression~\eqref{eq:f_V in C_2} as $k\to 1$ when $p=p(k)=\frac{3}{2}K(k)\to\infty$.
	Since $\E(p)\to 1$ as $k\to 1$, there exists some large enough $\tilde{k}<1$ such that, whenever $\tilde{k}<k<1$, we have the bounds
	\begin{equation*}
	-2p\leq g_2(p) \leq -\frac{1}{2}p.
	\end{equation*}
	Note that for all $0<k<1$, we also have
	\begin{align}\label{eq-elliptic-signs}
	&\cn p<0,&
	&0<\sn p<1,&
	&0<\dn p.
	\end{align}
	
	Using the above estimates, we obtain bounds for the various parts of expression~\eqref{eq:f_V in C_2}. Namely, for all $\tilde{k}<k<1$, we have
	\begin{align*}
	8 k^2 \cn^2 p  \sn^2 p &\geq 0,\\
	g_2^2(p) (3 - 6 \sn^2 p) & \geq -12p^2,\\ 
	g_2^3(p) - k^4 p & \leq -\frac{1}{8}p^3,\\
	-2 g_2(p) (4 + k^2 (1 - 6 \sn^2 p)) & \leq 20p.
	\end{align*}
	With the above four inequalities along with the sign information of \eqref{eq-elliptic-signs}, we deduce the lower bound
	\begin{equation}\label{eq-f2v-lower-bound}
	f_2^V(p,k) \geq -16p^2\dn p
	- \frac{1}{6}p^3\cn p \sn p
	+ \frac{80}{3}p\cn p \sn p.
	\end{equation}
	
	For $p=\frac{3}{2}K(k)$, we have the explicit expressions
	\begin{align*}
	\dn p &= (1-k^2)^{1/4},\\
	\cn p &= -\frac{(1-k^2)^{1/4}}{\sqrt{1+\sqrt{1-k^2}}}
	= -\frac{\dn p}{\sqrt{1+\dn^2 p}}.
	\end{align*}
	Since $\dn p\to 0$, $\sn p\to 1$ and $p\to\infty$ as $k\to 1$, we may further assume (increasing $\tilde{k}$ if necessary) that, for $\tilde{k}<k<1$, we have the additional bounds
	\begin{align*}
	\frac{1}{2}\dn p \leq -\cn p &\leq \dn p,\\
	\frac{1}{2} \leq \sn p &\leq 1,\\
	\frac{1}{24}p^3 - 16p^2 - \frac{80}{3}p&\geq 0.
	\end{align*}
	With these extra conditions, we conclude from \eqref{eq-f2v-lower-bound} that
	\begin{align*}
	f_2^V(p,k) &\geq - 16p^2\dn p +\frac{1}{12}p^3 \dn p \sn p-\frac{80}{3}p\dn p\sn p
	\\&\geq \dn p \Big(\frac{1}{24}p^3 - 16p^2 - \frac{80}{3}p \Big)
	\geq 0
	\end{align*}
	for all $\tilde{k}<k<1$.
\end{proof}

\bibliographystyle{amsalpha}
\bibliography{biblio}
\end{document}